\documentclass[reqno,centertags, 12pt]{amsart}
\usepackage{amsmath,amsthm,amscd,amssymb}
\usepackage{latexsym}
\newcommand{\bbC}{{\mathbb{C}}}
\newcommand{\bbD}{{\mathbb{D}}}

\newcommand{\bbR}{{\mathbb{R}}}

\newcommand{\calC}{{\mathcal{C}}}

\newcommand{\calM}{{\mathcal M}}
\newcommand{\calP}{{\mathcal P}}

\newcommand{\lb}{\label}
\newcommand{\f}{\frac}

\newcommand{\ol}{\overline}
\newcommand{\ti}{\tilde  }

\newcommand{\cvh}{\text{\rm{cvh}}}
\newcommand{\tr}{\text{\rm{Tr}}}

\newcommand{\dist}{\text{\rm{dist}}}

\newcommand{\ran}{\text{\rm{ran}}}

\newcommand{\ess}{\text{\rm{ess}}}

\newcommand{\s}{\text{\rm{s}}}

\newcommand{\supp}{\text{\rm{supp}}}
\newcommand{\intt}{\text{\rm{int}}}

\newcommand{\bi}{\bibitem}

\newcommand{\beq}{\begin{equation}}
\newcommand{\eeq}{\end{equation}}
\newcommand{\ba}{\begin{align}}
\newcommand{\ea}{\end{align}}
\newcommand{\veps}{\varepsilon}
\let\det=\undefined\DeclareMathOperator{\det}{det}

\newcounter{smalllist}
\newenvironment{SL}{\begin{list}{{\rm\roman{smalllist})}}{%
\setlength{\topsep}{0mm}\setlength{\parsep}{0mm}\setlength{\itemsep}{0mm}%
\setlength{\labelwidth}{2em}\setlength{\leftmargin}{2em}\usecounter{smalllist}%
}}{\end{list}}

\DeclareMathOperator*{\wlim}{w-lim}

\allowdisplaybreaks
\numberwithin{equation}{section}
\newtheorem{theorem}{Theorem}[section]

\newtheorem*{p2.1}{Proposition 2.1}
\newtheorem{proposition}[theorem]{Proposition}
\newtheorem{lemma}[theorem]{Lemma}

\theoremstyle{definition}

\theoremstyle{remark}
\newtheorem*{remark}{Remark}
\newtheorem*{remarks}{Remarks}

\newcommand{\abs}[1]{\lvert#1\rvert}
\newcommand{\jap}[1]{\langle #1 \rangle}

\begin{document}
\title{Weak Convergence of CD Kernels and Applications}

\author[B.~Simon]{Barry Simon$^*$}

\thanks{$^3$ Mathematics 253-37, California Institute of Technology, Pasadena, CA 91125.
E-mail: bsimon@caltech.edu. Supported in part by NSF grant DMS-0140592 and
U.S.--Israel Binational Science Foundation (BSF) Grant No.\ 2002068}

\date{July 11, 2007}
\keywords{CD kernel, orthogonal polynomials, weak convergence}
\subjclass[2000]{Primary: 33C45; Secondary: 60B10, 05E35}

\begin{abstract} We prove a general result on equality of the weak limits of the zero
counting measure, $d\nu_n$, of orthogonal polynomials (defined by a measure $d\mu$) and
$\f{1}{n} K_n (x,x)\, d\mu(x)$. By combining this with M\'at\'e--Nevai and Totik upper
bounds on $n\lambda_n(x)$, we prove some general results on $\int_I \f{1}{n} K_n(x,x)\,
d\mu_\s\to 0$ for the singular part of $d\mu$ and $\int_I \abs{\rho_E(x) - \f{w(x)}{n}
K_n(x,x)}\, dx\to 0$, where $\rho_E$ is the density of the equilibrium measure and
$w(x)$ the density of $d\mu$.
\end{abstract}

\maketitle

%%%%%%%%%%%%%%%%%%%%%%%%%%%%%%%
\section{Introduction} \lb{s1}
%%%%%%%%%%%%%%%%%%%%%%%%%%%%%%%

We will discuss here orthogonal polynomials on the real line (OPRL) and unit circle (OPUC)
(see \cite{Szb,GBk,FrB,OPUC1,OPUC2,Rice}). $d\mu$ will denote a measure on $\partial\bbD=\{z\in
\bbC\mid \abs{z}=1\}$ (positive but not necessarily normalized), $\Phi_n(z,d\mu)$ and $\varphi_n
(z,d\mu)$ its monic and normalized orthogonal polynomials, and $\{\alpha_n\}_{n=0}^\infty$
its Verblunsky coefficients determined by $(\Phi_n^*(z)\equiv z^n \,\ol{\Phi_n(1/\bar z)}$)
\begin{equation} \lb{1.1}
\Phi_{n+1}(z) =z\Phi_n(z) - \bar\alpha_n \Phi_n^*(z)
\end{equation}
and
\begin{equation} \lb{1.2}
\|\Phi_n\|_{L^2 (\partial\bbD,d\mu)} =\mu(\partial\bbD)^{1/2} \prod_{j=0}^{n-1}
(1-\abs{\alpha_j}^2)^{1/2}
\end{equation}

$d\mu$ will also denote a measure on $\bbR$ of compact support, $P_n(x,d\mu)$ and $p_n(x,d\mu)$
its monic and normalized orthogonal polynomials. $\{a_n,b_n\}_{n=1}^\infty$ are its Jacobi
parameters defined by
\begin{equation} \lb{1.3}
xp_n(x)=a_{n+1} p_{n+1}(x) + b_n p_n(x) + a_n p_{n-1}(x)
\end{equation}
and
\begin{equation} \lb{1.4}
\|P_n\|_{L^2(\bbR,d\mu)} = \mu(\bbR)^{1/2} (a_1\dots a_n)
\end{equation}

The CD kernel is defined by (some authors sum only to $n-1$)
\begin{align}
K_n(z,w) &=\sum_{j=0}^n \ol{\varphi_n(z)}\, \varphi_n (w) \lb{1.5} \\
K_n (x,y) &=\sum_{j=0}^n p_j(x) p_j(y) \lb{1.6}
\end{align}

The Lebesgue decomposition
\begin{align}
d\mu(e^{i\theta}) &= w(\theta)\, \f{d\theta}{2\pi} + d\mu_\s (e^{i\theta}) \lb{1.6a} \\
d\mu(x) &= w(x)\, dx + d\mu_\s (x) \lb{1.6b}
\end{align}
with $d\mu_\s$ Lebesgue singular will enter.

To model the issues that concern us here, we recall two consequences of the Szeg\H{o}
condition for OPUC, namely,
\begin{equation} \lb{1.6x}
\int \log(w(\theta))\, \f{d\theta}{2\pi} > -\infty
\end{equation}

Here are two central results:

\begin{theorem}[Szeg\H{o}, 1920 \cite{Sz20}]\lb{T1.1} If the Szeg\H{o} condition holds, then
\begin{equation} \lb{1.7}
\lim_{n\to\infty}\int \abs{\varphi_n (e^{i\theta})}^2\, d\mu_\s =0
\end{equation}
\end{theorem}

\begin{remark} In distinction, if $w=0$, $\int \abs{\varphi_n (e^{i\theta})}^2\, d\mu_\s
\equiv 1$.
\end{remark}

\begin{theorem}[M\'at\'e--Nevai--Totik \cite{MNT91}]\lb{T1.2} If the Szeg\H{o} condition
holds, then for a.e.\ $\theta$,
\begin{equation} \lb{1.8}
w(\theta)\, \f{1}{n+1} \, K_n (e^{i\theta},e^{i\theta}) \to 1
\end{equation}
\end{theorem}

They state the result in an equivalent form involving the Christoffel function
\begin{equation} \lb{1.9a}
\lambda_n(z_0) =\inf \biggl\{\int \abs{Q_n(e^{i\theta})}^2\, d\mu(\theta)\biggm|
\deg Q_n\leq n; \, Q_n(z_0)=1\biggr\}
\end{equation}
The minimizer is
\begin{align}
Q_n (e^{i\theta}) &= \f{K_n (z_0, e^{i\theta})}{K_n(z_0,z_0)} \lb{1.9} \\
\lambda_n(z_0) &= K_n (z_0, z_0)^{-1} \lb{1.10}
\end{align}
and since \eqref{1.6x} implies $w(\theta)>0$ for a.e.\ $\theta$, \eqref{1.8} is equivalent to
\begin{equation} \lb{1.11}
n \lambda_{n-1} (e^{i\theta}) \to w(\theta)
\end{equation}

It is also known under a global $w(\theta) >0$ condition that there are similar results that
come from what is called Rakhmanov theory (see \cite[Ch.~9]{OPUC2} and references therein):

\begin{theorem}\lb{T1.3} If $w(\theta) >0$ for a.e.\ $\theta$, then
\begin{SL}
\item[{\rm{(i)}}]  \eqref{1.7} holds.
\item[{\rm{(ii)}}]
\begin{equation} \lb{1.9x}
\lim_{n\to\infty} \int \bigl|w(\theta) \abs{\varphi_n (e^{i\theta})}^2-1\bigr|\, \f{d\theta}{2\pi} =0
\end{equation}
\end{SL}
\end{theorem}

\begin{remark} (i) is due to Rakhmanov \cite{Rakh77,Rakh83}; (ii) is due to M\'at\'e, Nevai, and Totik
\cite{MNT88}.
\end{remark}

Theorems~\ref{T1.1} and \ref{T1.2} are known to hold under a local Szeg\H{o} condition together
with regularity in the sense of Stahl--Totik \cite{StT} (see also \cite{EqMC} and below).

One of our goals here is to prove the first results of these genres with neither a local Szeg\H{o}
condition nor a global a.c.\ condition. While we focus on the OPRL case, for comparison with
the above theorems, here are our new results for OPUC:

\begin{theorem}\lb{T1.4} If $I$ is an interval on $\partial\bbD$ so that
\begin{SL}
\item[{\rm{(a)}}] $w(\theta)>1$ a.e.\ on $I$
\item[{\rm{(b)}}] $d\mu$ is regular for $\partial\bbD$, that is,
\begin{equation} \lb{1.13}
(\rho_1 \dots \rho_n)^{1/n} \to 1
\end{equation}
\end{SL}
then
\begin{alignat}{2}
& \text{\rm{(i)}} \qquad && \int_I \f{1}{n+1} \, K_n (e^{i\theta},e^{i\theta})\,
d\mu_\s (\theta) \to 0 \lb{1.11x} \\
& \text{\rm{(ii)}} \qquad && \int_I \, \biggl| 1-w(\theta)\, \f{1}{n+1} K_n
(e^{i\theta},e^{i\theta})\biggr|\, \f{d\theta}{2\pi} \to 0 \lb{1.12x}
\end{alignat}
\end{theorem}

\begin{remarks} 1. We do not have pointwise convergence \eqref{1.8}, but do have
``one-half" of it, namely, for $e^{i\theta}\in I$,
\begin{equation} \lb{1.13x}
\liminf\, w(\theta)\, \f{1}{n+1}\, K_n (e^{i\theta},e^{i\theta})\to 1
\end{equation}

\smallskip
2. One associates existence of limits of $\abs{\varphi_n}^2\, d\mu$ with Rakhmanov
which we only expect when $d\mu$ has support on all of $\partial\bbD$ or a single
interval of $\bbR$. At best, with multiple intervals, one expects almost periodicity
of $\abs{\varphi_n}^2\, d\mu$ rather than existence of the limits. For this reason,
the Ces\`aro averages of Theorem~\ref{T1.4} are quite natural.
\end{remarks}

There is nothing sacred about $\partial\bbD$---regularity is defined for any set, and for
both OPRL and OPUC, all we need is regularity plus $w(\theta) >0$ on an interval. We also
have interesting new bounds on the density of zeros when regularity fails---these generalize
a theorem of Totik--Ullman \cite{TU}.

These new theorems do not involve tweaking the methods used to prove Theorems~\ref{T1.1}--\ref{T1.3},
but a genuinely new technique (plus one general method of M\'at\'e--Nevai used in the proof of
Theorem~\ref{T1.2}). Our point here is as much to emphasize this new technique as to prove
the results. The new technique is the following:

\begin{theorem}\lb{T1.5} Let $d\mu$ be a measure on $\bbR$ with bounded support or a
measure on $\partial\bbD$. For $\bbR$, let $d\nu_n$ be the normalized zero counting measure
for the OPRL and for $\partial\bbD$ for the zeros of the paraorthogonal polynomials
{\rm{(}}POPUC{\rm{)}}. Let $n(j)$ be a subsequence $n(1) < n(2) < \dots$. Then
\begin{equation} \lb{1.14}
d\nu_{n(j)+1} \overset{w}{\longrightarrow} d\nu_\infty \Longleftrightarrow \f{1}{n(j)+1}\,
K_{n(j)}(x,x)\, d\mu(x) \overset{w}{\longrightarrow} d\nu_\infty
\end{equation}
\end{theorem}

\begin{remarks} 1. We will discuss POPUC and related objects on $\partial\bbD$ in Section~\ref{s2}.

\smallskip
2. As we will discuss, $d\nu_{n(j)}$ and $d\nu_{n(j)+1}$ have the same limits.
\end{remarks}

In one sense, this result is more than twenty-five years old! It is a restatement of the invariance
of the density of states under change of boundary conditions proven in this context first by
Avron--Simon \cite{AvS83}. But this invariance is certainly not usually stated in these terms.
We also note that for OPUC, I noted this result in my book (see \cite[Thm.~8.2]{OPUC2}) but did
not appreciate its importance.

We note that for OPUC, limits of $\f{1}{n+1} K_n (e^{i\theta},e^{i\theta})\, d\mu(\theta)$ have
been studied by Golinskii--Khrushchev \cite{KhGo} without explicitly noting the connection to
CD kernels. Their interesting results are limited to the case of OPUC and mainly to situations
where the support is all of $\partial\bbD$.

By itself, Theorem~\ref{T1.4} is interesting (e.g., it could be used to streamline the proof
of a slightly weaker version of Corollary~2 of Totik \cite{Tot}), but it is really powerful
when used with the following collection of results:

\begin{theorem}[M\'at\'e--Nevai \cite{MN}]\lb{T1.6} For any measure on $\partial\bbD$,
\begin{equation} \lb{1.15}
\limsup\, n\lambda_{n-1} (e^{i\theta})\leq w(\theta)
\end{equation}
\end{theorem}

M\'at\'e--Nevai--Totik \cite{MNT91} noted that this applies to OPRL on $[-1,1]$ using the
Szeg\H{o} mapping, and Totik, in a brilliant paper \cite{Tot}, shows how to extend it to any
measure of bounded support, $E$, on $\bbR$ so long as $E$ contains an interval:

\begin{theorem}[Totik \cite{Tot}]\lb{T1.7} Let $I\subset E\subset\bbR$ where $I=(a,b)$ is an
interval and $E$ is compact. Suppose $d\mu$ is a measure with support contained in $E$ so that
\begin{equation} \lb{1.16}
d\mu(x)=w(x)\, dx + d\mu_\s (x)
\end{equation}
Let $d\rho_E(x)$ be the potential theoretic equilibrium measure for $E$ {\rm{(}}so it is known
$d\rho_E\restriction I=\rho_E(x)\, dx$ for some strictly positive, real analytic weight
$\rho_E${\rm{)}}. Then for Lebesgue a.e.\ $x\in I$,
\begin{equation} \lb{1.17}
\limsup \, n\lambda_{n-1}(x) \leq \f{w(x)}{\rho_E(x)}
\end{equation}
\end{theorem}

\begin{remarks} 1. Totik concentrated on the deeper and more subtle fact that if there is a
local Szeg\H{o} condition on $I$ and $\mu$ is regular for $E$, then the limit exists and equals
$w(x)/\rho_E(x)$ for a.e.\ $x$. But along the way he proved \eqref{1.17}.

\smallskip
2. We will actually prove equality in \eqref{1.17} (for $\limsup$) when $E=\supp(d\mu)$ and
$\mu$ is regular.

\smallskip
3. Later (see Section~8), we will prove the analog of Theorem~\ref{T1.7} for closed subsets
of $\partial\bbD$.
\end{remarks}

Some of our results assume regularity of $\mu$ so we briefly summarize the main results from
that theory due largely to Stahl--Totik \cite{StT} in their book; see my recent paper
\cite{EqMC} for an overview.

A measure $\mu$ on $\bbR$ is called {\it regular\/} if $E=\supp(d\mu)$ is compact and
\begin{equation} \lb{1.18}
\lim_{n\to\infty}\, (a_1\dots a_n)^{1/n} =C(E)
\end{equation}
where $C(E)$ is the (logarithmic) capacity of $E$. For OPUC, \eqref{1.18} is replaced by
\begin{equation} \lb{1.19}
\lim_{n\to\infty}\, (\rho_1 \dots \rho_n)^{1/n} =C(E)
\end{equation}
Some insight is gained if one knows in both cases, for any $\mu$ (supported on $E$ compact
in $\bbR$ or $\partial\bbD$), the $\limsup$ is bounded above by $C(E)$. For this paper,
regularity is important because of (this is from \cite[Sect.~2.2]{StT} or
\cite[Thm.~2.5]{EqMC}).

\begin{theorem}\lb{T1.8} If $d\mu$ is regular, then with $d\nu_n$, the density of zeros of
the OPRL or of the POPUC, we have
\begin{equation} \lb{1.20}
d\nu_n \overset{w}{\longrightarrow} d\rho_E
\end{equation}
the equilibrium measure for $E$. Conversely, if \eqref{1.20} holds, either $\mu$ is regular
or $\mu$ is supported on a set of capacity zero.
\end{theorem}

We should mention one criterion for regularity that goes back to Erd\"os--Tur\'an \cite{ET} for
$[-1,1]$ and Widom \cite{Wid} for general $E$:

\begin{theorem}\lb{T1.9} If $E=\supp(d\mu)$ and
\[
d\mu(x) =f(x)\, d\rho_E(x) + d\mu_\s (x)
\]
where $\mu_\s$ is $\rho_E$-singular and $f(x)>0$ for $\rho_E$-a.e.\ $x$, then $\mu$
is regular.
\end{theorem}

There is a proof of Van Assche \cite{VA} of the general case presented in \cite{EqMC} that is
not difficult. But in the case where $E^{\intt}$ (interior in the sense of $\bbR$) differs
from $E$ by a set of capacity zero (e.g., $E=[-1,1]$), we will find a proof that uses
only our Theorem~\ref{T1.5}/\ref{T1.7} strategy. In particular, we will have a proof of
the Erd\"os--Tur\'an result that uses neither potential theory nor polynomial inequalities.

We can now describe the content of this paper. In Section~\ref{s2}, we prove the main
weak convergence result, Theorem~\ref{T1.5}. In Section~\ref{s3}, we prove the analog of
the Erd\"os--Tur\'an result for $\partial\bbD$ and illustrate how these ideas are connected
to regularity criterion of Stahl--Totik \cite{StT}. Section~\ref{s4} proves Theorem~\ref{T1.4}
for $\mu$ regular on $\partial\bbD$. Sections~\ref{s5} and \ref{s6} then parallel
Sections~\ref{s3} and \ref{s4} but for general compact sets $E\subset\bbR$. Section~\ref{s7},
motivated by work of Totik--Uhlmann \cite{TU}, provides a comparison result about densities
of zeros. Section~\ref{s9} does the analog of Sections~\ref{s5} and \ref{s6} for general
$E\subset\partial\bbD$. To do this, we need Totik's result \eqref{1.17} in that situation.
This does not seem to be in the literature, so Section~\ref{s8} fills that need.

\medskip
It is a pleasure to thank Jonathan Breuer, Yoram Last, and especially Vilmos Totik for useful
conversations. I would also like to thank Ehud de Shalit and Yoram Last for the hospitality of the
Einstein Institute of Mathematics of the Hebrew University during part of the preparation of
this paper.

%%%%%%%%%%%%%%%%%%%%%%%%%%%%%%%
\section{Weak Convergence} \lb{s2}
%%%%%%%%%%%%%%%%%%%%%%%%%%%%%%%

Our main goal here is to prove a generalization of Theorem~\ref{T1.5} (and so also that theorem). We
let $\mu$ be a measure of compact support in $\bbC$ and
\begin{equation} \lb{2.1}
N(\mu)=\sup\{\abs{z}\mid z\in\supp(d\mu)\}
\end{equation}
We let $M_z$ be multiplication by $z$ on $L^2 (\bbC,d\mu)$, so
\begin{equation} \lb{2.2}
\|M_z\|=N(\mu)
\end{equation}
and let $Q_n$ be the $(n+1)$-dimensional orthogonal projection onto polynomials of degree $n$ or
less. All estimates here depend on

\begin{proposition} \lb{P2.1} Fix $\ell=1,2,\dots$. Then
\begin{equation} \lb{2.3}
Q_n M_z^\ell Q_n - (Q_n M_z Q_n)^\ell
\end{equation}
is an operator of rank at most $\ell$ and norm at most $2N(\mu)^\ell$.
\end{proposition}

\begin{proof} Let $X$ be the operator in \eqref{2.3}. Clearly, $X=0$ on $\ran(1-Q_n) = \ker(Q_n)$.
Since $M_z$ maps $\ran (Q_j)$ to $\ran (Q_{j+1})$, $X=0$ on $\ran (Q_{n-\ell})$. This shows that
$X$ has rank at most $\ell$. The norm estimate is immediate from \eqref{2.2}.
\end{proof}

Here is the link to $K_n$ and to $d\nu_{n+1}$. Let $X_j(z,d\mu)$ be the monic OPs for $\mu$ and
let $x_n = X_n/\|X_n\|$.

\begin{proposition}\lb{P2.2}
\begin{SL}
\item[{\rm{(i)}}] We have for all $w\in\bbC$,
\begin{equation} \lb{2.4}
\det_{Q_n}(w-Q_n M_z Q_n) =X_{n+1} (w,d\mu)
\end{equation}
In particular, if $d\nu_{n+1}$ is the zero counting measure for $X_{n+1}$, then
\begin{equation} \lb{2.5}
\f{1}{n+1} \tr ((Q_n M_z Q_n)^\ell) =\int z^\ell\, d\nu_{n+1}(z)
\end{equation}

\item[{\rm{(ii)}}] Let
\begin{equation} \lb{2.6}
K_n (z,w) =\sum_{j=0}^n \ol{x_j(z)}\, x_j(w)
\end{equation}
Then
\begin{equation} \lb{2.7}
\tr(Q_n M_z^\ell Q_n) = \int z^\ell K_n(z,z)\, d\mu(z)
\end{equation}
\end{SL}
\end{proposition}

\begin{remark} The proof of (i) is due to Davies and Simon \cite{DS}.
\end{remark}

\begin{proof} (i)  Let $z_0$ be a zero of $X_{n+1}$ of order $\ell$. Let $\varphi(z)
=X_{n+1}(z)/(z-z_0)^\ell$. Since $Q_n [X_{n+1}]=0$, we see, with $M_z^{(n)}= Q_n M_z Q_n$,
\begin{equation} \lb{2.8x}
(M_z^{(n)}-z_0)^\ell \varphi=0 \qquad
(M_z^{(n)} -z_0)^{\ell-1} \varphi\neq 0
\end{equation}
showing that $z_0$ is a zero of $\det_{Q_n}(w-M_z^{(n)})$ of order at least $\ell$. In this way, we
see $X_{n+1}(w)$ and the $\det$ have the same zeros. Since both are monic, we obtain \eqref{2.4}.

In particular, this shows
\[
\tr((M_z^{(n)})^\ell) =\sum_{\substack{\text{zeros $z_j$ of} \\ \text{multiplicity $m_j$}}}\,
m_j z_j^\ell =(n+1)\int z^\ell\, d\nu_{n+1}(z)
\]
proving \eqref{2.5}

\smallskip
(ii) In $L^2 (\bbC,d\mu)$, $\{x_j\}_{j=0}^n$ span $\ran (Q_n)$ and are an orthonormal basis, so
\begin{align*}
\tr(Q_n M_z^\ell Q_n) &= \sum_{j=0}^n \jap{x_j, z^\ell x_j} \\
&= \int \sum_{j=0}^n z^\ell \abs{x_j(z)}^2\, d\mu(z)
\end{align*}
proving \eqref{2.7}.
\end{proof}

\begin{proposition}\lb{P2.3} Let $d\eta_n$ be the probability measure
\begin{equation} \lb{2.8}
d\eta_n(z) = \f{1}{n+1}\, K_n(z,z)\, d\mu(z)
\end{equation}
Then for $\ell=0,1,2,\dots$,
\begin{equation} \lb{2.9}
\biggl| \int z^\ell\, d\eta_n(z) - \int z^\ell\, d\nu_{n+1}(z)\biggr| \leq
\f{2\ell N(\mu)^\ell}{n+1}
\end{equation}

Suppose there is a compact set, $K\subset\bbC$, containing the supports of all $d\nu_n$ and
the support of $d\mu$ so that $\{z^\ell\}_{\ell=0}^\infty \cup \{\bar z^\ell\}_{\ell=0}^\infty$
are $\|\cdot\|_\infty$-total in the continuous function on $K$. Then for any subsequence
$n(j)$, $d\eta_{n(j)}\to d\nu_\infty$ if and only if $d\nu_{n(j)+1}\to d\nu_\infty$.
\end{proposition}

\begin{proof} \eqref{2.9} is immediate from \eqref{2.5} and \eqref{2.7} if we note that, by
Proposition~\ref{P2.1},
\begin{equation} \lb{2.10}
\abs{\tr(Q_n M_z^\ell Q_n)-\tr ((Q_n M_z Q_n)^\ell)}\leq 2N(\mu)^\ell \ell
\end{equation}
\eqref{2.9} in turn implies that we have the weak convergence result.
\end{proof}

While we were careful to use $n(j)$ and $n(j)+1$, we note that since $\abs{x_j(z)}^2\, d\mu$ is
a probability measure, we have
\begin{equation} \lb{2.10a}
\|\eta_n-\eta_{n+1}\| \leq \f{1}{n+1} + n \biggl(\f{1}{n}-\f{1}{n+1}\biggr) \leq \f{2}{n+1}
\end{equation}
so we could just as well have discussed weak limits of $\eta_{n(j)+1}$ and $\nu_{n(j)+1}$. Here
is Theorem~\ref{T1.5} for OPRL:

\begin{theorem}\lb{T2.4} For OPRL, $d\eta_{n(j)}$ converges weakly to $d\nu_\infty$ if and only if
$d\nu_{n(j)+1}$ converges weakly to $d\nu_\infty$.
\end{theorem}

\begin{proof} $\{x_j\}_{j=0}^\infty$ are total in $C([\alpha,\beta])$ for any real interval
$[\alpha,\beta]$, so Proposition~\ref{P2.3} is applicable.
\end{proof}

For OPUC, we need a few preliminaries: Define $\calP\colon C(\partial\bbD)\to C(\bbD)$
(with $\bbD=\{z\mid\abs{z}\leq 1\}$) by
\begin{equation} \lb{2.11}
(\calP f)(re^{i\theta}) =\int \f{1-r^2}{1+r^2- 2r\cos (\theta-\varphi)} \, f(e^{i\theta})\,
\f{d\varphi}{2\pi}
\end{equation}
for $r<1$ and $(\calP f)(e^{i\theta})=f(e^{i\theta})$ and $\calP^*\colon\calM_{+,1}(\bbD)\to
\calM_{+,1}(\partial\bbD)$, the balayage, by duality. Then the following is well known and
elementary:

\begin{proposition}\lb{P2.5} $\calP^*(d\nu)$ is the unique measure, $\eta$, on $\partial\bbD$ with
\begin{equation} \lb{2.12}
\int e^{i\ell\theta}\, d\eta(\theta) = \int z^\ell\, d\nu(z)
\end{equation}
for $\ell=0,1,2,\dots$.
\end{proposition}

Widom \cite{Wid} proved that if $\mu$ is supported on a strict subset of $\partial\bbD$, then
the zero counting measure, $d\nu_n$, has weak limits supported on $\partial\bbD$,
\begin{equation} \lb{2.13}
\calP^*(d\nu_n) -d\nu_n\overset{w}{\longrightarrow} 0
\end{equation}

Finally, we note the following about paraorthogonal polynomials (POPUC) defined for $\beta\in
\partial\bbD$ by
\begin{equation} \lb{2.14}
P_{n+1}(z,\beta) =z\Phi_n(z)-\bar\beta \Phi_n^*(z)
\end{equation}
defined in \cite{JNT89} and studied further in \cite{CMV02,CMV,CMV06,Gol02,S308,Wong}.

\begin{proposition}\lb{P2.6} Let $\beta_n$ be an arbitrary sequence $\partial\bbD$ and let
$d\nu_{n+1}^{(\beta_n)}$ be the zero counting measure for $P_{n+1}(z,\beta_n)$ {\rm{(}}known
to live on $\partial\bbD$; see, e.g. \cite{S308}{\rm{)}}. Then for any $\ell\geq 0$,
\begin{equation} \lb{2.15}
\biggl| \int z^\ell\, d\nu_{n+1} -\int z^\ell \, d\nu_{n+1}^{(\beta_n)}\biggr| \to 0
\end{equation}
\end{proposition}

\begin{proof} Let $\calC_{n+1,F}$ be the truncated CMV matrix of size $n+1$ whose
eigenvalues are the zeros of $\Phi_{n+1}(z)$ (see \cite[Ch.~4]{OPUC1}) and let
$\calC_{n+1,F}^{(\beta_n)}$ be the unitary dialation whose eigenvalues are the
zeros of $P_{n+1}(z,\beta_n)$ (see \cite{CMV06,S308,S-46}). Then $\calC_{n+1,F} -
\calC_{n+1,F}^{(\beta_n)}$ is rank one with norm bounded by $2$, so $\calC_{n+1,F}^\ell
- [\calC_{n+1,F}^{(\beta)}]^\ell$ is rank at most $\ell$ with norm $2$. Thus
\begin{equation} \lb{2.16}
\abs{\tr (\calC_{n+1,F}^\ell - (\calC_{n+1,F}^{(\beta)})^\ell)} \leq 2\ell
\end{equation}
and
\[
\abs{\text{LHS of \eqref{2.15}}} \leq \f{2\ell}{n+1}
\]
proving \eqref{2.15}.
\end{proof}

\begin {theorem}\lb{T2.7} For OPUC, $d\eta_{n(j)}$ converges weakly to $d\nu_\infty$ if and only if
each of the following converges weakly to $d\nu_\infty$:
\begin{SL}
\item[{\rm{(i)}}] $d\nu_{n+1}^{(\beta_n)}$ for any $\beta_n$.
\item[{\rm{(ii)}}] $\calP (d\nu_{n+1})$, the balayage of $d\nu_{n+1}$.
\item[{\rm{(iii)}}] $d\nu_{n+1}$ if $\supp(d\mu)\neq\partial\bbD$.
\end{SL}
\end{theorem}

\begin{proof} Immediate from Proposition~\ref{P2.3} and the above remark.
\end{proof}

When I mentioned Theorem~\ref{T2.4} to V.~Totik, he found an alternate proof, using tools more
familiar to OP workers, that provides some insight. With his permission, I include this proof.
Let me discuss the result for convergence of sequences rather than subsequences and then give
some remarks to handle subsequences.

The key to Totik's proof is Gaussian quadratures which says that if $\{x_j^{(n)}\}_{j=1}^n$
are the zeros of $p_n(x,d\mu)$, and $\lambda_{n-1}(x)=K_{n-1}(x,x)^{-1}$, then for any polynomial
$R_m$ of degree $m\leq 2n-1$,
\begin{equation} \lb{2.17}
\int R_m(x)\, d\mu(x) =\sum_{j=1}^n \lambda_{n-1} (x_j^{(n)}) R_m (x_j^{(n)})
\end{equation}
We will also need the Christoffel variation principle
\begin{equation} \lb{2.18}
n\leq q\Rightarrow \lambda_q (x) \leq \lambda_n(x)
\end{equation}
which is immediate from \eqref{1.9a}.

Suppose we know $d\nu_n\to d\nu_\infty$. Fix a polynomial $Q_m$ of degree $m$ with $Q_m\geq 0$ on
$\cvh(\supp(d\mu))$. Let
\[
R_{m+2n}(x) = \f{1}{n+1}\, Q_m(x) K_n(x,x)
\]
which has degree $m+2n$. Thus, if $N=n+m$, we have by \eqref{2.17} that
\begin{align}
\int Q_m(x) &\biggl[\f{1}{n+1}\, K_n(x,x)\biggr]\, d\mu \notag \\
&= \sum_{j=1}^N \lambda_{N-1} (x_j^{(N-1)}) Q_m (x_j^{(N-1)})\biggl(\f{1}{n+1}\biggr)
\lambda_n (x_j^{(N-1)})^{-1} \lb{2.19} \\
&\leq \biggl( \f{N+1}{n+1}\biggr) \f{1}{N+1} \sum_{j=1}^N Q_m (x_j^{(N-1)}) \lb{2.20}
\end{align}
by $Q_m\geq 0$ and \eqref{2.19} (so $\lambda_{N-1}/\lambda_n\leq 1$). As $n\to\infty$, $N/n\to 1$.
So, by hypothesis,
\begin{equation} \lb{2.21}
\text{RHS of \eqref{2.20}} \to \int Q_m(y)\, d\nu_\infty (y)
\end{equation}

We conclude by \eqref{2.19} that
\begin{equation} \lb{2.22}
\limsup \int Q_m(x) \, d\eta_n(x) \leq \int Q_m(y)\, d\nu_\infty (y)
\end{equation}
If $0\leq Q_m\leq 1$ on $\cvh(\supp(d\mu))$, we can apply this also to $1-Q_m$ and so conclude
for such $Q_m$ that
\[
\lim\int Q_m\, d\eta_n(x) = \int Q_m\, d\nu_\infty
\]
which implies $\wlim d\eta_n=d\nu_\infty$.

The same inequality \eqref{2.20} can be used to show that if $d\eta_n\overset{w}{\longrightarrow}
d\eta_\infty$, then
\[
\liminf \int Q_m(y)\, d\nu_m(y) \geq \int Q_m(y)\, d\eta_\infty (y)
\]
and thus, by the same $1-Q$ trick, we get $d\nu_n\to d\eta_\infty$.

To handle subsequences, we only need to note that, by \eqref{2.10a}, if $d\eta_{n(j)}\to d\eta_\infty$,
then $d\eta_{n(j)+\ell}\to d\eta_\infty$ for $\ell=0,\pm 1, \pm 2, \dots$. Similarly, by zero
interlacing, if $d\nu_{n(j)}\to d\nu_\infty$, then $d\nu_{n(j)+\ell}\to d\nu_\infty$ for fixed $\ell$.

By using the operator theoretic proof of Gaussian quadrature (see, e.g., \cite[Sect.~1.2]{OPUC1}), one
sees this proof is closely related to our proof above.

%%%%%%%%%%%%%%%%%%%%%%%%%%%%%%%%%%%%%%%%%%%%%%%%%%%%%%%%%%%%%%%%%%%%%%%%%%%%%%
\section{Regularity for $\partial\bbD$: The Erd\"os--Tur\'an Theorem} \lb{s3}
%%%%%%%%%%%%%%%%%%%%%%%%%%%%%%%%%%%%%%%%%%%%%%%%%%%%%%%%%%%%%%%%%%%%%%%%%%%%%%

Our goal in this section is to prove:

\begin{theorem}\lb{T3.1} Let $d\mu$ on $\partial\bbD$ have the form \eqref{1.6a} with $w(\theta)>0$
for a.e.\ $\theta$. Then $\mu$ is regular, that is, \eqref{1.19} holds with $E=\partial\bbD$
{\rm{(}}so $C(E)=1${\rm{)}}.
\end{theorem}

\begin{remarks} 1. This is an analog of a theorem for $[-1,1]$ proven by Erd\"os and Tur\"an
\cite{ET}. Our proof here seems to be new.

\smallskip
2. This is, of course, weaker than Rakhmanov's theorem (see \cite[Ch.~9]{OPUC2} and references therein),
but as we will see, this extends easily to some other situations.
\end{remarks}

The proof will combine the M\'at\'e--Nevai theorem (Theorem~\ref{T1.6}) and Proposition~\ref{P2.5}.
It is worth noting where the M\'at\'e--Nevai theorem comes from. By \eqref{1.9a}, if
\begin{equation} \lb{3.1}
Q_n (e^{i\theta}) = \f{1}{n+1} \sum_{j=0}^n e^{ij(\theta-\varphi)}
\end{equation}
then
\begin{equation} \lb{3.2}
\lambda_n (e^{i\varphi}) \leq \int \abs{Q_n (e^{i\theta})}^2\, d\mu(\theta)
\end{equation}
Recognizing $(n+1)\abs{Q_n}^2$ as the Fej\'er kernel, \eqref{1.15} is a standard maximal function a.e.\
convergence result.

\begin{proposition}\lb{P3.2} For any measure on $\partial\bbD$ for a.e.\ $\theta$,
\begin{equation} \lb{3.3}
\liminf_{n\to\infty}\, \f{1}{n+1}\, K_n (e^{i\theta},e^{i\theta})\geq w(\theta)^{-1}
\end{equation}
On the set where $w(\theta) >0$,
\begin{equation} \lb{3.4}
\liminf_{n\to\infty}\, \f{1}{n+1}\, w(\theta) K_n (e^{i\theta},e^{i\theta}) \geq 1
\end{equation}
\end{proposition}

\begin{remark} If $w(\theta) =0$, \eqref{3.3} is interpreted as saying that the limit is infinite.
In that case, of course, \eqref{3.4} does not hold.
\end{remark}

\begin{proof} \eqref{3.3} is immediate from \eqref{1.10} and \eqref{1.15}. If $w\neq 0$,
$ww^{-1}=1$, so \eqref{3.3} implies \eqref{3.4}.
\end{proof}

\begin{proof}[Proof of Theorem~\ref{T3.1}] The hypothesis $w>0$ for a.e.\ $\theta$ implies $d\mu$
is not supported on a set of capacity zero. Thus, by Theorem~\ref{T1.8}, regularity holds if we
prove that the density of zeros of POPUC converges to $\f{d\theta}{2\pi}$. By Proposition~\ref{P2.5},
this follows if we prove that $\f{1}{n+1} K_{n+1}\, d\mu\to\f{d\theta}{2\pi}$.

Suppose $n(j)\to\infty$ is a subsequence with $\f{1}{n+1} K_{n+1} (e^{i\theta}, e^{i\theta})
w(\theta)\f{d\theta}{2\pi}\to d\nu_1$ and $\f{1}{n+1} K_{n+1} (e^{i\theta}, e^{i\theta})
d\mu_\s\to d\nu_2$. Then since $\f{1}{n+1} K_n\, d\mu$ is normalized,
\begin{equation} \lb{3.5}
\int [d\nu_1 + d\nu_2]=1
\end{equation}
On the other hand, by Fatou's lemma and \eqref{3.4} for any continuous $f\geq 0$ and
the hypothesis $w(\theta) >0$ a.e.\ $\theta$,
\begin{align}
\int f\, d\nu_1 &= \lim \int f \biggl[ \f{1}{n+1}\, w(\theta) K_n (e^{i\theta},e^{i\theta})\biggr]
\f{d\theta}{2\pi} \notag \\
&\geq \int \liminf \biggl[ f\, \f{1}{n+1}\, w(\theta) K_n (e^{i\theta},e^{i\theta})\biggr]
\f{d\theta}{2\pi} \notag \\
&\geq \int f(\theta)\, \f{d\theta}{2\pi} \lb{3.6}
\end{align}
Thus,
\begin{equation} \lb{3.7}
d\nu_1 \geq \f{d\theta}{2\pi}
\end{equation}
By \eqref{3.5}, this can only happen if
\begin{equation} \lb{3.8}
d\nu_1 =\f{d\theta}{2\pi} \qquad d\nu_2 =0
\end{equation}

Compactness of the space of measures proves that $\f{1}{n+1} K_{n+1}\, d\mu\to d\theta/2\pi$.
\end{proof}

Along the way, we also proved $d\nu_2=0$, that is,

\begin{theorem}\lb{T3.3} Under the hypotheses of Theorem~\ref{T3.1},
\[
\f{1}{n+1} \int  \sum_{j=0}^n \abs{\varphi_j (e^{i\theta})}^2\, d\mu_\s (\theta)\to 0
\]
\end{theorem}

It would be interesting to see if these methods provide an alternate proof of the theorem of
Stahl--Totik \cite{StT} that, if for all $\eta >0$,
\begin{equation} \lb{3.9}
\lim_{n\to\infty}\, \bigl| \bigl\{\theta\mid \mu (\{\psi\mid \abs{e^{i\theta} - e^{i\psi}} \leq
\tfrac{1}{n}\}) \leq e^{-n\eta}\bigr\}\bigr| =0
\end{equation}
then $\mu$ is regular. The point is that using powers of the Fej\'er kernel, one can get
trial functions localized in an interval of size $O(\f{1}{n})$, and off a bigger interval of
size $O(\f{1}{n})$, it is exponentially small. \eqref{3.9} should say that the dominant
contribution comes from an $O(\f{1}{n})$ interval. On the other hand, the translates of
these trial functions are spread over $O(\f{1}{n})$ intervals, so one gets lower bounds on
$\f{1}{n} K_n (e^{i\theta},e^{i\theta})$ of order $\mu (e^{i\theta} - \f{c}{n}, e^{i\theta} +
\f{c}{n})^{-1}$ which are then integrated against $d\mu$ canceling this inverse and hopefully
leading to \eqref{3.7}, and so regularity.

%%%%%%%%%%%%%%%%%%%%%%%%%%%%%%%%%%%%%%%%%%%%%%%%%%
\section{Localization on $\partial\bbD$} \lb{s4}
%%%%%%%%%%%%%%%%%%%%%%%%%%%%%%%%%%%%%%%%%%%%%%%%%%

In this section, we will prove Theorem~\ref{T1.4}. Instead of using the M\'at\'e--Nevai bound to
prove regularity, we combine it with regularity to get information.

\begin{proof}[Proof of Theorem~\ref{T1.4}] As in the proof of Theorem~\ref{T3.1}, we let $d\nu_1,
d\nu_2$ be weak limits of $\f{1}{n} K(e^{i\theta}, e^{i\theta}) w(\theta) \f{d\theta}{2\pi}$ and
$\f{1}{n} K_n (e^{i\theta},e^{i\theta}) d\mu_\s (\theta)$. On $I$, the same arguments as above imply
\begin{equation} \lb{4.1}
d\nu_1 \restriction I \geq \f{d\theta}{2\pi} \restriction I
\end{equation}

By regularity, globally
\begin{equation} \lb{4.2}
d\nu_1 + d\nu_2 = \f{d\theta}{2\pi}
\end{equation}
Thus, on $I$,
\begin{equation} \lb{4.3}
\nu_1 \restriction I = \f{d\theta}{2\pi} \qquad \nu_2 \restriction I=0
\end{equation}

The second implies \eqref{1.11x}. The first implies that
\begin{equation} \lb{4.4}
\int_a^b w(\theta) \, \f{1}{n+1}\, K_n (e^{i\theta},e^{i\theta})\, \f{d\theta}{2\pi} \to (b-a)
\end{equation}

Since \eqref{3.4} and Fatou imply
\begin{equation} \lb{4.5}
\int_a^b \biggl[ w(\theta)\, \f{1}{n+1}\, K_n (e^{i\theta},e^{i\theta})-1 \biggr]_- \,
\f{d\theta}{2\pi} \to 0
\end{equation}
we obtain \eqref{1.12x}.
\end{proof}

Notice that \eqref{4.4} and \eqref{3.4} imply a pointwise a.e.\ result (which we stated as \eqref{1.13x})
\begin{equation} \lb{4.6}
\liminf_{n\to\infty}\, \f{1}{n+1}\, w(\theta) K_n (e^{i\theta},e^{i\theta}) =1
\end{equation}
We do not know how to get a pointwise result on $\limsup$ under only the condition $w(\theta) >0$
(but without a local Szeg\H{o} condition).

%%%%%%%%%%%%%%%%%%%%%%%%%%%%%%%%%%%%%%%%%%%%%%%%%%%%%%%%%%%%%%%%
\section{Regularity for $E\subset\bbR$: Widom's Theorem} \lb{s5}
%%%%%%%%%%%%%%%%%%%%%%%%%%%%%%%%%%%%%%%%%%%%%%%%%%%%%%%%%%%%%%%%

In this section and the next, our goal is to extend the results of the last two sections to situations
where $\partial\bbD$ is replaced by fairly general closed sets in $\bbR$. The keys will be Theorem~\ref{T2.4}
and Totik's Theorem~\ref{T1.7}. We begin with a few remarks on where Theorem~\ref{T1.7} comes from
(see also Section~\ref{s8}).

Since the hypothesis is that $\supp(d\mu)\subset E$, not $=E$, it suffices to find $E_n\supset E$ so that
$\rho_{E_n}(x)\to \rho_E(x)$ on $I$ and for which \eqref{1.14} can be proven. By using $\ti E_n=\{x\mid
\dist(x,E)\leq \f{1}{n}\}$, one first gets approximation by a finite union of intervals and then, by a
theorem proven by Bogatyr\"ev \cite{Bog}, Peherstorfer \cite{Peh}, and Totik \cite{Tot-acta}, one finds
$\ti E_n\subset E_n$ where the $E_n$'s are finite unions of intervals with rational harmonic measure.
For rational harmonic measures, one can use as trial polynomials $K_m(x,x_0)/K_m(x_0,x_0)$ where
$K_m$ is the CD kernel of a measure in the periodic isospectral torus and Floquet theory. (This is the
method from Simon \cite{2exts}; Totik \cite{Tot} instead uses polynomial mapping.)

\begin{theorem}\lb{T5.1} Let $E\subset\bbR$ be a compact set with $\partial E\equiv E\setminus E^\intt$
{\rm{(}}$E^\intt$ means interior in $\bbR${\rm{)}} having capacity zero {\rm{(}}e.g., a finite version
of closed intervals{\rm{)}}. Let $d\mu$ be a measure with $\sigma_\ess (d\mu)=E$ and
\begin{equation}\lb{5.1}
d\mu = f(x)\, d\rho_E + d\mu_\s
\end{equation}
where $d\mu_\s$ is $d\rho_E$-singular. Suppose $f(x)>0$ for $d\rho_E$-a.e.\ $x$. Then $d\mu$ is regular.
\end{theorem}

\begin{remarks} 1. In this case, $d\rho_E$ is equivalent to $\chi_E \, dx$.

\smallskip
2. For any compact $E$, this is a result of Widom \cite{Wid}; see also Van Assche \cite{VA}, Stahl--Totik
\cite{StT}, and Simon \cite{EqMC}.
\end{remarks}

\begin{proof} Essentially identical to Theorem~\ref{T3.1}. By Theorem~\ref{T1.8} and the fact that $d\mu$
is clearly not supported on sets of capacity zero, it suffices to prove that $d\nu_n\to d\rho_E$. Pick
$n(j)\to\infty$ so $\f{1}{n(j)+1} K_{n(j)}(x,x) f(x)\, d\rho_E$ and $\f{1}{n(j)+1} K_{n(j)}(x,x)\,
d\mu_\s$ separately have limits $d\nu_1$ and $d\nu_2$. \eqref{1.14} says that (given that
$f(x) >0$ for a.e.\ $x$)
\begin{equation}\lb{5.2}
\liminf\, \f{1}{n(j)+1}\, K_{n(j)}(x,x) f(x) \geq 1
\end{equation}
By Fatou's lemma on $E^\intt$,
\begin{equation}\lb{5.3}
d\nu_1 \geq d\rho_E
\end{equation}

Since also $\int (d\nu_1+d\nu_2) = 1$ and $\int_{E^\intt} d\rho_E =1$ (since $C(E\setminus E^\intt) =0$),
we conclude $d\nu_1=d\rho_E$, $d\nu_2 =0$. By compactness of probability measures, $\f{1}{n(j)+1} K_{n(j)}
(x,x)\, d\mu\overset{w}\longrightarrow d\rho_E$, implying regularity.
\end{proof}

%%%%%%%%%%%%%%%%%%%%%%%%%%%%%%%%%%%%%%%%%
\section{Localization on $\bbR$} \lb{s6}
%%%%%%%%%%%%%%%%%%%%%%%%%%%%%%%%%%%%%%%%%

Here is an analog of Theorem~\ref{T1.4} for any $E\subset\bbR$.

\begin{theorem}\lb{T6.1} Let $I=[a,b]\subset E\subset\bbR$ with $a<b$ and $E$ compact. Let $d\mu$ be a measure
on $\bbR$ so that $\sigma_\ess (d\mu)=E$ and $\mu$ is regular for $E$. Suppose
\begin{equation}\lb{6.1}
d\mu = w(x)\, dx+d\mu_\s
\end{equation}
with $d\mu_\s$ Lebesgue singular. Suppose $w(x) >0$ for a.e.\ $x\in I$. Then
\begin{SL}
\item[{\rm{(i)}}] $\f{1}{n+1} K_n(x,x)\, d\mu_\s \overset{w}{\longrightarrow} 0$
\item[{\rm{(ii)}}] $\int_I \abs{\rho_E(x) - \f{1}{n+1} w(x) K_n (x,x)}\, dx \to 0$
\end{SL}
\end{theorem}

\begin{proof} By $w(x)>0$ a.e.\ on $I$ and \eqref{1.13},
\[
\liminf \, w(x)\, \f{1}{n+1}\, K_n(x,x)\geq \rho_E(x)
\]
for a.e.\ $x\in I$. From this, one can follow exactly the proofs in Section~\ref{s4}.
\end{proof}

%%%%%%%%%%%%%%%%%%%%%%%%%%%%%%%%%%%%%%%%%%%%%%%%%%%%
\section{Comparisons of Density of Zeros} \lb{s7}
%%%%%%%%%%%%%%%%%%%%%%%%%%%%%%%%%%%%%%%%%%%%%%%%%%%%

In \cite{TU}, Totik--Ullman proved the following (we take $[-a,a]$ rather than $[a,b]$ only for
notational simplicity):

\begin{theorem}[\cite{TU}] \lb{T7.1} Let $d\mu$ be a measure supported on a subset of $[-1,1]$ where
\begin{equation}\lb{7.1}
d\mu =w(x)\, dx + d\mu_\s
\end{equation}
with
\[
w(x) >0 \quad{\text{for a.e. }} x\in [-a,a]
\]
for some $a\in (0,1)$. Let $d\nu_\infty$ be any limit point of the zero counting measures for $d\mu$.
Then on $(-a,a)$, we have that
\begin{equation}\lb{7.2}
(2\pi)^{-1} (1-x^2)^{-1/2}\, dx \leq d\nu_\infty(x) \leq (2\pi)^{-1} (a^2-x^2)^{-1/2}\, dx
\end{equation}
\end{theorem}

Our goal here is to prove the following, which we will show implies Theorem~\ref{T7.1} as
a corollary:

\begin{theorem}\lb{T7.2} Let $d\mu_1, d\mu_2$ be two measures on $\bbR$ of compact support.
Suppose that for some interval $I=(\alpha,\beta)$,
\begin{alignat}{2}
&\text{\rm{(i)}} \qquad && d\mu_1 \leq d\mu_2 \lb{7.3} \\
&\text{\rm{(ii)}} \qquad && d\mu_1 \restriction (\alpha,\beta) = d\mu_2 \restriction (\alpha,\beta) \lb{7.4}
\end{alignat}
Let $n(j)\to\infty$ and suppose $d\nu_{n(j)}^{(k)}\to d\nu_\infty^{(k)}$ for $k=1,2$, where $d\nu_n^{(k)}$
is the zero counting measure for $d\mu_k$. Then on $(\alpha,\beta)$,
\begin{equation}\lb{7.5}
d\nu_\infty^{(2)} \restriction (\alpha,\beta) \leq d\nu_\infty^{(1)} \restriction (\alpha,\beta)
\end{equation}
\end{theorem}

\begin{proof} By \eqref{1.9a},
\begin{equation}\lb{7.6}
\lambda_n (x,d\mu_1) \leq \lambda_n (x,d\mu_2)
\end{equation}
so, by \eqref{1.10},
\begin{equation}\lb{7.7}
\f{1}{n+1}\, K_n^{(2)} (x,x) \leq \f{1}{n+1} \, K_n^{(1)} (x,x)
\end{equation}
for all $x$. By \eqref{7.4} on $(\alpha,\beta)$,
\begin{equation}\lb{7.8}
\f{1}{n+1}\, K_n^{(2)} (x,x) \, d\mu_2 \leq \f{1}{n+1}\, K_n^{(1)}(x,x)\, d\mu_1
\end{equation}
By Theorem~\ref{T2.4}, this implies \eqref{7.5}.
\end{proof}

\begin{proof}[Proof of Theorem~\ref{T7.1}] Let $d\mu_2 = d\mu$ and let $d\mu_1 = \chi_{(-a,a)}[d\mu]$ so
$d\mu_1\leq d\mu_2$ with regularity on $[-a,a]$. By Theorem~\ref{T5.1}, $d\mu_1$ is regular for
$[a,a]$, so $d\nu_n^{(2)}\to (2\pi)^{-1} (a^2 -x^2)^{-1/2}\, dx$, the equilibrium measure for $[-a,a]$.
Thus \eqref{7.5} implies the second inequality in \eqref{7.2}.

On the other hand, let $d\mu_1=d\mu$ and let $d\mu_2 = [\chi_{(-1,1)}-\chi_{(-a,a)}]\, dx+d\mu$. Then
$d\mu_1\leq d\mu_2$ with equality on $(-a,a)$. Moreover, $d\mu_2$ is regular for $[-1,1]$ by
Theorem~\ref{T5.1} and the hypothesis $\sigma(d\mu)\subset [-1,1]$. Thus, $d\nu_n^{(1)}\to (2\pi)^{-1}
(1-x^2)^{-1/2}\, dx$. Thus \eqref{7.5} implies the first inequality in \eqref{7.2}.
\end{proof}

\begin{remarks} 1. Theorem~\ref{T7.1} only requires $\sigma_\ess (d\mu)\subset [-1,1]$.

\smallskip
2. The theorems in \cite{TU} are weaker than Theorem~\ref{T7.1} in one respect and stronger in another.
They are weaker in that, because of their dependence on potential theory, they require that one of the
comparison measures be regular. On the other hand, they are stronger in that our reliance on weak
convergence limits us to open sets like $(\alpha,\beta)$, while they can handle more general sets.

\smallskip
3. \cite{EqMC} has an example of a measure, $d\mu$, on $[-1,1]$ where $w(x) >0$ on $[-1,0]$ and
the zero counting measures include among its limit points the equilibrium measures for $[-1,0]$
and for $[-1,1]$. This shows in the $[a,b]$-form of Theorem~\ref{T7.1}, both inequalities in
\eqref{7.2} can be saturated!
\end{remarks}

%%%%%%%%%%%%%%%%%%%%%%%%%%%%%%%%%%%%%%%%%%%%
\section{Totik's Bound for OPUC} \lb{s8}
%%%%%%%%%%%%%%%%%%%%%%%%%%%%%%%%%%%%%%%%%%%%

As preparation for applying our strategy to subsets of $\partial\bbD$, we need to prove an analog of
Totik's bound \eqref{1.16} for closed sets on $\partial\bbD$. Given $a,b\in\partial\bbD$, we let $I=
(a,b)$ be the ``interval" of all points ``between" $a$ and $b$, that is, going counterclockwise from $a$
to $b$, so $-1\in (e^{i\theta}, e^{i(2\pi -\theta)})$ for $0<\theta <\pi$ but $-1\notin (e^{i(2\pi-\theta)},
e^{i\theta})$. Given $E\subset\partial\bbD$ closed, we let $d\rho_E$ be its equilibrium measure. If $I\subset
E\subset\partial\bbD$ is a nonempty open interval, then
\begin{equation} \lb{8.1}
d\rho_E \restriction I = \rho_E (\theta)\, dm (\theta)
\end{equation}
where $dm= d\theta/2\pi$. The main theorem of this section is

\begin{theorem}\lb{T8.1} Let $I\subset E\subset\partial\bbD$ where $I=(a,b)$ is an interval and $E$
is closed. Let $d\mu$ be a measure with support in $E$ so that
\begin{equation} \lb{8.2}
d\mu(\theta) =w(\theta)\, dm+ d\mu_\s (\theta)
\end{equation}
Then for $dm$-a.e.\ $\theta\in I$, we have
\begin{equation} \lb{8.3}
\limsup\, n\lambda_{n-1}(e^{i\theta}) \leq \f{w(\theta)}{\rho_E(\theta)}
\end{equation}
\end{theorem}

Following Totik's strategy \cite{Tot,Tot-acta} for OPRL, we do this in two steps:

\begin{theorem}\lb{T8.2} \eqref{8.3} holds if $\supp(d\mu)\subset E^\intt$ and $E$ is a finite union
of intervals whose relative harmonic measures are rational.
\end{theorem}

\begin{theorem}\lb{T8.3} For any closed $E$ in $\partial\bbD$, we can find $E_n$ with
\begin{SL}
\item[{\rm{(i)}}] Each $E_n$ is a finite union of intervals whose relative harmonic measures
are rational.
\item[{\rm{(ii)}}]
\begin{equation} \lb{8.4}
E\subset E_n^\intt
\end{equation}
\item[{\rm{(iii)}}]
\begin{equation} \lb{8.5}
C(E_n\setminus E)\to 0
\end{equation}
\end{SL}
\end{theorem}

\begin{remark} If $E=I_1 \cup\cdots\cup I_\ell$, the relative harmonic measures are $\rho_E(I_j)$
which sum to $1$.
\end{remark}

\begin{proof}[Proof of Theorem~\ref{T8.1} given Theorems~\ref{T8.2} and \ref{8.3}] By general
principles, since $I$ is an interval, $\rho_{E_n}(\theta)$ and $\rho_E(\theta)$ are real analytic
with bounded derivatives. \eqref{8.5} implies $d\rho_{E_n}\overset{w}{\longrightarrow}d\rho_E$
and then the bounded derivative implies $\rho_{E_n}(\theta) \to \rho_E(\theta)$ uniformly on
compact subsets of $I$. By Theorem~\ref{T8.2}, LHS of \eqref{8.3} $\leq w(\theta)/\rho_{E_n}(\theta)$
for each $n$. Since $\rho_{E_n}(\theta)\to \rho_E(\theta)$, we obtain \eqref{8.3}.
\end{proof}

Our proof of Theorems~\ref{T8.2} and \ref{T8.3} diverges from the Totik strategy in two ways.
He obtains Theorem~\ref{T8.2} by using polynomial maps. Instead, following Simon \cite{2exts},
we use Floquet solutions.

Second, Totik shows if $E$ has $\ell$ gaps, one can find $E_n$ with rational relative harmonic
measures also with $\ell$ gaps obeying \eqref{8.4} and \eqref{8.5}. This is a result with rather
different proofs by Bogatyr\"ev \cite{Bog}, Peherstorfer \cite{Peh}, and Totik \cite{Tot-acta}.
I believe any of these proofs will extend to OPUC, but we will settle for a weaker result---our
$E_n$'s will contain up to $2\ell$ intervals, the first $\ell$ each containing one of the $\ell$
intervals of $E$ and an additional $\ell$ or fewer exponentially small intervals. This will
allow us to ``get away" with following only the easier part of Peherstorfer's strategy.

\begin{proof}[Proof of Theorem~\ref{T8.2}] By Theorem~11.4.5 of \cite{OPUC2}, $E$ is the essential
spectrum of an isospectral torus of Verblunsky coefficients periodic up to a phase, that is,
for suitable $p$ (chosen so that $p\rho_E(I_j)$ is an integer for each $j$),
\[
\alpha_{n+1} =\lambda\alpha_n
\]
Let $\mu_E$ be the measure associated to a point on the isospectral torus.

By Floquet theory (see \cite[Sect.~11.2]{OPUC2}), for $z\in E^\intt$, $\varphi_n$ is a sum of
two functions each periodic up to a phase and, by \cite[Sect.~11.12]{OPUC2}, on compact subsets,
$K$, of $E^\intt$,
\begin{equation} \lb{8.6}
\sup_{z\in K,n}\, \abs{\varphi_n(z)} <\infty
\end{equation}
It follows from the Christoffel--Darboux formula (see \cite[Sect.~2.2]{OPUC1}) that
\begin{equation} \lb{8.7}
\sup_{z,w\in K}\, \abs{K_n (z,w)}\leq C\abs{z-w}^{-1}
\end{equation}

The almost periodicity of $\varphi$ implies uniformly on $K$, $\f{1}{n+1} K_n (z,z)$ has a finite
nonzero limit, and then, by Theorem~\ref{T2.7}, the limit must be $\rho_E(\theta)/w_E(\theta)$,
that is, uniformly on $K$,
\begin{equation} \lb{8.8}
\lim_{n\to\infty}\, \f{1}{n}\, K_n (e^{i\theta},e^{i\theta}) = \f{\rho_E(\theta)}{w_E(\theta)}
\end{equation}
where $w_E(\theta)$ is the weight for $d\rho_E(\theta)$. Moreover, as proven in Simon \cite{2exts},
for any $A>0$, uniformly on $e^{i\theta}\in K$ and $n\abs{\varphi-\theta}<A$,
\begin{equation} \lb{8.9}
Q_n (e^{i\varphi}) \equiv \f{K_n(e^{i\theta},e^{i\varphi})}{K_n (e^{i\theta},e^{i\theta})}
= \f{\sin (n\rho_E(\theta)(\theta-\varphi))}{n(\theta-\varphi)\rho_E(\theta)}\, (1+O(1))
\end{equation}

Now use $Q_n (e^{i\varphi})$ as a trial function in \eqref{1.9a}. By \eqref{8.7} and \eqref{8.8}, by
taking $A$ large, the contribution of $n\abs{\varphi-\theta}>A$ can be made arbitrarily small. Maximal
function arguments and \eqref{8.9} show that the contribution of the region $n\abs{\varphi-\theta}<A$
to $n\lambda_{n-1}$ is close to $w(\theta)/\rho_E(\theta)$. This proves \eqref{8.3} for $d\mu$.
\end{proof}

Given $E\subset\partial\bbD$ compact, define
\begin{equation} \lb{8.10}
\ti E_n = \{e^{i\theta}\in\bbD\mid \dist (e^{i\theta}, E)\leq \tfrac{1}{n}\}
\end{equation}
It is easy to see that $C(\ti E_n\setminus E)\to 0$ and $\ti E_n$ is a union of $\ell(n)<\infty$
closed intervals. It thus suffices to prove Theorem~\ref{T8.3} when $E$ is already a union of finitely
many $\ell$ disjoint closed intervals, and it is that we are heading towards. (Parenthetically, we note
that we could dispense with this and instead prove the analog of Theorem~\ref{T8.2} for a finite
union of intervals using Jost solutions for the isospectral torus associated to such finite gap sets,
as in Simon \cite{2exts}.)

Define $\calP_n$ to be the set of monic polynomials all of whose zeros lie in $\partial\bbD$. Since
\[
e^{i\theta/2} + e^{i\varphi} e^{-i\theta/2} = e^{i\varphi/2} [e^{i(\theta-\varphi)/2} +
e^{-i(\theta-\varphi)/2}]
\]
if $P\in\calP_n$, there is a phase factor $e^{i\eta}$ so
\begin{equation} \lb{8.10x}
z^{-n/2} e^{i\eta} P(z)\text{ is real on } \partial\bbD
\end{equation}
Define the restricted Chebyshev polynomials, $\ti T_n$, associated to $E\subset\partial\bbD$ by requiring
that $\ti T_n$ miminize
\begin{equation} \lb{8.11}
\|P_n\|_E =\sup_{z\in E}\, \abs{P_n(z)}
\end{equation}
over all $P_n\in\calP_n$. We will show that for $n$ large, $2\ti T_{2n}/\|\ti T_{2n}\|_E$ are the
rotated discriminants associated to sets $E_n$ that approximate an $E$ which is a finite union of
intervals. An important input is

\begin{lemma}\lb{L8.4} Let $I=(z_0,z_1)$ be an interval in $\partial\bbD$. For any small $\varphi$, let
$I_\varphi=(z_0 e^{i\varphi}, z_1 e^{-i\varphi})$, so for $\varphi >0$, $I_\varphi$ is smaller than $I$.
For $\varphi <0$,
\[
\abs{(z-z_0 e^{i\varphi})(z-z_1 e^{-i\varphi})}
\]
decreases on $\partial\bbD\setminus I$ and increases on $I$ as $\varphi$ decreases in $(-\veps, 0)$.
\end{lemma}

\begin{proof} Take $z_1=\bar z_0$ and then use some elementary calculus.
\end{proof}

\begin{theorem}\lb{T8.5} Let $E$ be a finite union of disjoint closed intervals on $\partial\bbD$,
$E=I_1\cup \cdots \cup I_\ell$ and let $\partial\bbD\setminus E=G_1\cup \cdots\cup G_\ell$ has
$\ell$ gaps. Then
\begin{SL}
\item[{\rm{(i)}}] Each $\ti T_n$ has at most one zero in each $G_\ell$.
\item[{\rm{(ii)}}] If $z_j^{(n)}$ is the zero of $\ti T_n$ in $G_j$, then on any compact
$K\subset G_j$, we have
\begin{equation} \lb{8.12}
\lim_{n\to\infty} \, \inf_{z\in K} \biggl( \biggl| \f{\ti T_n(z)}{(z-z_j^{(n)})\|F_n\|_E}
\biggr| \,\biggr)^{1/n} >1
\end{equation}
\item[{\rm{(iii)}}] At any local maximum, $\ti z$, of $\abs{\ti T_n(z)}$ in some $I_j$, we have
\begin{equation} \lb{8.13}
\abs{\ti T_n(\ti z)} = \|\ti T_n\|_E
\end{equation}
\end{SL}
\end{theorem}

\begin{proof} (i) If there are two zeros in some $G_j$, we can symmetrically move the zeros apart.
Doing that increases $\ti T_n$ on $G_j$ which is disjoint from $E$, but it decreases $\|\ti T_n\|_E$,
contradicting the minimizing definition.

\smallskip
(ii) The zero counting measure for $\ti T_n$ converges to the equilibrium measure on $E$. For
standard $T_n$'s and $E\subset\bbR$, this result is proven in \cite{AB,ST,EqMC}. A small change
implies this result for $\ti T_n$. This, in turn, says that uniformly on $G_j$,
\begin{equation} \lb{8.14}
\biggl| \f{\ti T_n(z)}{(z-z_j^{(n)})\|\ti T_n\|_E} \biggr|^{1/n} \to
\exp (-\Phi_{\rho_E}(z))
\end{equation}
which implies \eqref{8.12}.

\smallskip
(iii) Since $z^{-n/2}\ti T_n(z)$ is real up to a phase, the local maxima of $\abs{\ti T_n(z)}$ on
$\partial\bbD$ alternate with the zeros of $\ti T_n(z)$. If a local maximum is smaller than
$\|\ti T_n\|_E$, we move the nearest zeros, say $z_0,z_1$, apart. That decreases
$\|\ti T_n\|_{E\setminus (z_0,z_1)}$ and increases $\|\ti T_n\|_{(z_0,z_1)}$. Since the latter
is assumed smaller than $\|\ti T_n\|_E$, it decreases $\|\ti T_n\|_E$ overall, violating the
minimizing definition. Thus, $\abs{\ti T_n(z)}\geq \|\ti T_n\|_E$. But since $\ti z\in E$,
$\abs{\ti T_n(z)} \leq \|\ti T_n\|_E$.
\end{proof}

Now define
\begin{equation} \lb{8.15}
\Delta_n(z) = \f{2e^{i\varphi_n} \ti T_{2n}(z)}{\|\ti T_{2n}\|_E}
\end{equation}
where $\varphi_n$ is chosen to make $\Delta_n$ real on $\partial\bbD$. By \eqref{8.13}, maxima
in $E$ occur with $\Delta_n(z) =\pm 2$, and by \eqref{8.12}, maxima in $\partial\bbD\setminus E$
occur at points where $\abs{\Delta_n(z)} >2$. Thus, up to a phase, $\Delta_n(z)$ looks like a
discriminant. So, by Theorem~11.4.5 of \cite{OPUC2}, $E_n\equiv\Delta_n^{-1} ([-2,2])$ is
the essential spectrum of a CMV matrix whose Verblunsky coefficients obey $\alpha_{m+p}=\lambda
\alpha_m$ for $\abs{\lambda}=1$.

\begin{proof}[Proof of Theorem~\ref{T8.3}] $E_n$ has at most $2\ell$ components,
$\ell$ containing $I_1, \dots, I_\ell$ (call them $I_1^{(n)}, \dots, I_\ell^{(n)}$),
and $\ell$ possible components, $J_1^{(n)}, \dots, J_\ell^{(n)}$, one in each gap.
Since capacities are bounded by $\f14$ times Lebesgue measure, it suffices to show that
\[
\sum_{j=1}^\ell \, \abs{I_j^{(n)}\setminus I_j} + \abs{J_j^{(n)}}\to 0
\]
to prove \eqref{8.5} and complete the proof. Since, on $\partial\bbD$,
\begin{equation} \lb{8.16}
-\Phi_{\rho_E}(x) \geq c\, \dist (x,E)^{1/2}
\end{equation}
by \eqref{8.14}, we have $\abs{I_j^{(n)}\setminus I_j}\to 0$ and $\abs{J_j^{(n)}}\to 0$.
\end{proof}

%%%%%%%%%%%%%%%%%%%%%%%%%%%%%%%%%%%%%%%%%%%%%%%%%%%%%%%%%%%%%%
\section{Theorems for Subsets of $\partial\bbD$} \lb{s9}
%%%%%%%%%%%%%%%%%%%%%%%%%%%%%%%%%%%%%%%%%%%%%%%%%%%%%%%%%%%%%%

Given Theorem~\ref{T8.1} and our strategies in Sections~\ref{s3}--\ref{s6}, we
immediately have

\begin{theorem}\lb{T9.1} Let $E\subset\partial\bbD$ with $\partial E=E\setminus E^\intt$
{\rm{(}}$E^\intt$ means interior in $\partial\bbD${\rm{)}} having capacity zero. Let $d\mu$
be a measure on $\partial\bbD$ with $\sigma_\ess (d\mu)=E$ and
\begin{equation} \lb{9.1}
d\mu = f(x)\, d\rho_E + d\mu_\s
\end{equation}
where $d\mu_s$ is $d\rho_E$-singular. Suppose $f(x) >0$ for $d\rho_E$-a.e.\ $x$. Then
$d\mu$ is regular.
\end{theorem}

\begin{theorem}\lb{T9.2} Let $I\subset E\subset\partial\bbD$ with $I$ a nonempty closed
interval and $E$ closed. Let $d\mu$ be a measure on $\partial\bbD$ so $\sigma_\ess (d\mu)
=E$ and $\mu$ is regular for $E$. Suppose
\begin{equation} \lb{9.2}
d\mu = w(\theta)\, \f{d\theta}{2\pi} + d\mu_\s
\end{equation}
and $w(\theta)>0$ for a.e.\ $e^{i\theta}\in I$. Then
\begin{alignat*}{2}
&\text{\rm{(i)}} \qquad && \f{1}{n+1}\, K_n (e^{i\theta},e^{i\theta}) \, d\mu_\s (\theta)
\overset{w}{\longrightarrow} 0 \\
&\text{\rm{(ii)}} \qquad && \int_I \, \biggl| \rho_E(\theta) - \f{1}{n+1}\, w(\theta) K_n
(e^{i\theta},e^{i\theta})\biggr|\, \f{d\theta}{2\pi} \to 0
\end{alignat*}
where $d\rho(\theta) =\rho_E(\theta)\f{d\theta}{2\pi}$ on $I$.
\end{theorem}

\bigskip
%%%%%%%%%%%%%%%%%%%%%%%%%%%%%

\end{document}